\documentclass[11pt]{article}

\usepackage{amssymb, amsthm, amsmath, mathrsfs, fullpage}
\usepackage{tikz}
\usetikzlibrary{matrix}
\usepackage{enumerate}

\newcommand{\rs}{\Sigma^-}
\newcommand{\pihat}{\hat{\pi}}
\newcommand{\tent}{\Lambda}
\newcommand{\st}{\rho}
\renewcommand{\S}{\mathcal S}

\newcommand{\nn}[1]{\left\|#1\right\|^-}

\newcommand{\Allow}{\mathcal A}
\newcommand\C{\mathcal C}

\newcommand\hp{\hat{\pi}}
\newcommand\W{\mathcal W}

\newcommand\PP{\mathcal P}

\newcommand\lex{<_\mathrm{lex}}
\newcommand\I{\mathcal I}

\renewcommand{\S}{\mathcal{S}}

\DeclareMathOperator{\Pat}{Pat}

\theoremstyle{plain}
\newtheorem{theorem}{Theorem}[section]
\newtheorem{lemma}[theorem]{Lemma}

\newtheorem{corollary}[theorem]{Corollary}
\newtheorem{conjecture}[theorem]{Conjecture}

\theoremstyle{definition}

\newtheorem{sample result}[theorem]{Sample Result}
\newtheorem{example}[theorem]{Example}

\title{Characterization of the allowed patterns of signed shifts}
\author{Kassie Archer}
\date{}

\usepackage{setspace}
\begin{document}
\maketitle

\begin{abstract}
	The allowed patterns of a map are those permutations in the same relative order as the initial segments of orbits realized by the map. In this paper, we characterize and provide enumerative bounds for the allowed patterns of signed shifts, a family of maps on infinite words. 
\end{abstract}

\section{Introduction and Background}

Signed shifts are a family of maps on infinite words which generalize well-known maps such as the tent map and the left shift. The problem of characterizing the permutations which are realized by these maps have been studied in several papers including \cite{Amigosigned, Amigobook, AEK, Elishifts}. In \cite{Amigosigned}, the author presents a partial characterization of these permutations, called allowed patterns. In this paper, we show that the conditions presented in \cite{Amigosigned} are not sufficient for the permutations to be allowed and present a complete characterization of the allowed patterns of signed shifts. In Section \ref{sec:enum allowed}, we additionally provide bounds on the number of allowed patterns of size $n$ for each signed shift.

\subsection{Permutations}
We denote by $\S_n$ the set of permutations of $[n] = \{1, 2, \ldots, n\}$. We write permutations in one-line notation as $\pi=\pi_1\pi_2\dots\pi_n\in\S_n$. Occasionally, we will write a permutation in cycle notation as a product of disjoint cycles. A {\it cyclic permutation}, or {\it cycle}, is a permutation $\pi \in \S_n$ which is composed of a single $n$-cycle. We denote the set of cyclic permutations of length $n$ by $\C_n$. For example, the permutation $\pi = 37512864 = (13527684)$ is a cyclic permutation in $\C_8$ written in both its one-line notation and cycle notation.


It will be useful to define the map
$$\begin{array}{cccc}  \S_n & \to& \C_n\\
\pi &\mapsto &\hp,
\end{array}$$
where if $\pi = \pi_1\pi_2\dots \pi_n$ in one-line notation, then $\hp=(\pi_1,\pi_2,\dots,\pi_n)$ in cycle notation, that is, $\hp$ is the cyclic permutation that
sends $\pi_1$ to $\pi_2$, $\pi_2$ to $\pi_3$, and so on. Writing $\hp=\hp_1\hp_2\dots\hp_n$ in one-line notation, we have that $\hp_{\pi_i}=\pi_{i+1}$ for $1\le i\le n$, with the convention that $\pi_{n+1}:=\pi_1$.
The map $\pi\mapsto\hp$ also appears in~\cite{Elishifts}. 
For example, if $\pi = 17234856$, then $\pihat = (17234856) = 73486125$.


%


\subsection{Allowed patterns}
Let $X$ be a linearly ordered set and $x_1, x_2, \ldots, x_n \in X$ be distinct. Then we can define the {\it reduction operation} by $$\st(x_1,x_2, \ldots, x_n) = \pi$$ where $\pi = \pi_1\pi_2\ldots \pi_n$ is the permutation of $[n]$, written in one-line notation, whose entries are in the same relative order as the $n$ entries in the input. For example, $\st(3.3,3.7,9,6,0.2)=23541$. 

Consider a map $f: X \to X$. Iterating this map $f$ at a point $x\in X$ returns a sequence of elements from $X$ called the {\it orbit} of $x$ with respect to $f$: $$x, f(x), f^2(x), \ldots.$$ If there are no repetitions among the first $n$ elements of the orbit, then we define the {\it pattern} of $x$ with respect to $f$ of length $n$ to be 
$$\Pat(x, f, n) = \st(x, f(x), f^2(x), \dots, f^{n-1}(x)).$$ 
If $f^i(x)=f^j(x)$ for some $0\le i<j<n$, then $\Pat(x, f, n)$ is not defined. 
The set of {\em allowed patterns} of $f$ is the set of permutations which are realized by $f$ in this way:
$$\Allow(f)=\{\Pat(x,f,n):n\ge0,x\in X\}.$$ We denote by $\Allow_n(f)$ the allowed patterns of length $n$. Permutations which are not allowed patterns are called the {\it forbidden patterns} of $f$. 

For example, consider the logistic map $L$ on the unit interval defined by $L(x) = 4x(1-x).$ Then the pattern at $x = .3000$ of length 3 with respect to $L$ is the permutation $132$ since the first 3 elements of the orbit $.3000,.8400, .5376$ are in the same relative order as $132$. The allowed patterns of $L$ of length $3$ are $\Allow_3(f) = \{123, 132, 213, 231, 312\}$ and the forbidden patterns of $L$ of length 3 are $\S_3 \setminus \Allow_3(L) = \{321\}$, since there is no $x\in[0,1]$ so that $x,L(x), L(L(x))$ is in decreasing order. 

The set of allowed patterns is closed under consecutive pattern containment \cite{Elishifts} and the minimal forbidden patterns form a basis for the allowed patterns. These minimal forbidden patterns have been studied for various maps including logistic maps and signed shifts \cite{Eliu,Amigosigned}. 

It is known that if $f$ is a piecewise monotone map on the unit interval, then the size of $\Allow_n(f)$ grows at most exponentially \cite{BKP}, while the number of permutations grows super-exponentially and thus, $f$ will have forbidden patterns. The existence of forbidden patterns allows one to distinguish a random time series from a deterministic one \cite{AZS,AZS2}, since a random time series will eventually contain all patterns, while most patterns are forbidden in a deterministic time series. In addition, the size of $|\Allow_n(f)|$ for a given $f$ is known to be directly related to the topological entropy of $f$, a value which measures the complexity of the map \cite{BKP}. 


For these reasons, characterizing and enumerating the allowed and minimal forbidden patterns of a given map $f$ presents an interesting problem. Moreover, studying these ideas has also led to purely combinatorial results \cite{AE, Elishifts, Elicyc}. Previously, the question of characterizing and enumerating allowed patterns has been answered for the well-known left shift (called the $k$-shift) on words in \cite{Elishifts} and for $\beta$-shifts in \cite{Elibeta}. In this paper, we provide a characterization of the allowed patterns for the family of maps which called signed shifts, which generalize the $k$-shift and the well-known tent map, as well as bounds on the enumeration of these patterns. Though we do not approach the question of characterizing the forbidden (or minimal forbidden) patterns of signed shifts, this could be an interesting question for future study.

\subsection{Signed shifts}
\label{Sec:SS}


Let $k\ge2$ be fixed, and let $\W_k$ be the set of infinite words $s=s_1s_2\dots$ over the alphabet $\{0,1,\dots,k-1\}$. Let $\lex$ denote the lexicographic order on these words.
We use the notation $s_{[i,\infty)}=s_is_{i+1}\dots$, and $\bar{s_i}=k-1-s_i$. If $q$ is a finite word, $q^m$ denotes concatenation of $q$ with itself $m$ times, and $q^\infty$ is an {\it infinite periodic word}, defined as a word $s:=q^\infty$ so that $s = s_{[r+1,\infty)}$ where $|q| = r$.

Let the \textit{signature} of a signed shift be defined as $\sigma=\sigma_0\sigma_1\dots\sigma_{k-1} \in \{+,-\}^k$ and let $T^+_\sigma = \{t :  \sigma_t = +\}$ and $T^-_\sigma=\{ t  :  \sigma_t = -\}$. Note that $T^+_\sigma$ and $T^-_\sigma$ form a set partition of $\{0,1,\dots,k-1\}$. For example, if the signature of a signed shift is $\sigma = ++-+$, then $T^+_\sigma = \{0,1,3\}$ and $T^-_\sigma = \{2\}$. 
We give two definitions of the signed shift with signature $\sigma$, and show that they are order-isomorphic to each other. The first definition is the one commonly used in the literature and the second (equivalent) definition will be more convenient for studying the patterns realized by signed shifts. 

The first definition, which we denote by $\Sigma'_\sigma$, is the map $\Sigma'_\sigma:(\W_k,\lex)\to(\W_k,\lex)$ defined by
$$\Sigma'_\sigma(s_1s_2s_3s_4\dots)=\begin{cases} s_2s_3s_4\dots & \mbox{if }s_1\in T^+_\sigma, \\
\bar{s_2}\bar{s_3}\bar{s_4}\dots & \mbox{if }s_1\in T^-_\sigma. \end{cases}$$
For an example of a pattern under this map, let $\sigma = +--$ and $s = 001102012211\ldots$. The the pattern of $s$ with respect to the map $\Sigma_\sigma'$ of length 8 will be $\st(s, \Sigma'_\sigma(s), \Sigma'^2_\sigma(s), \ldots, \Sigma'^7_\sigma(s))$ where the relative order of the words is determined by the lexicographical ordering.
Thus, $\Pat(s, \Sigma'_\sigma,8) = 12453786.$

The order-preserving transformation
$$\begin{array}{cccc}  \phi_k:&(\W_k,\lex) &\to&([0,1],<) \\
&s_1s_2s_3s_4 \dots &\mapsto& \sum_{i\ge0}  s_i k^{-i-1}
\end{array}$$
can be used to show (see~\cite{Amigosigned}) that $\Sigma'_\sigma$ is order-isomorphic to the piecewise linear function, called the \textit{signed sawtooth map} with signature $\sigma$, $M_\sigma:[0,1]\to[0,1]$ defined
for $x \in [\frac{t}{k}, \frac{t+1}{k})$, for each $0\le t\le k-1$, as
$$M_\sigma(x)=\begin{cases} kx -t & \mbox{if } t \in T^+_\sigma, \\ t+1-kx & \mbox{if }t \in T^-_\sigma.\end{cases}$$
As a consequence, the allowed and forbidden patterns of $\Sigma'_\sigma$ are the same as those of $M_\sigma$, respectively.
A few examples of the graphs of the function $M_\sigma$ for different $\sigma$ are pictured in Figure~\ref{fig:signedshifts}.

\begin{figure}[ht]
\centering
   \includegraphics[width=.23\linewidth]{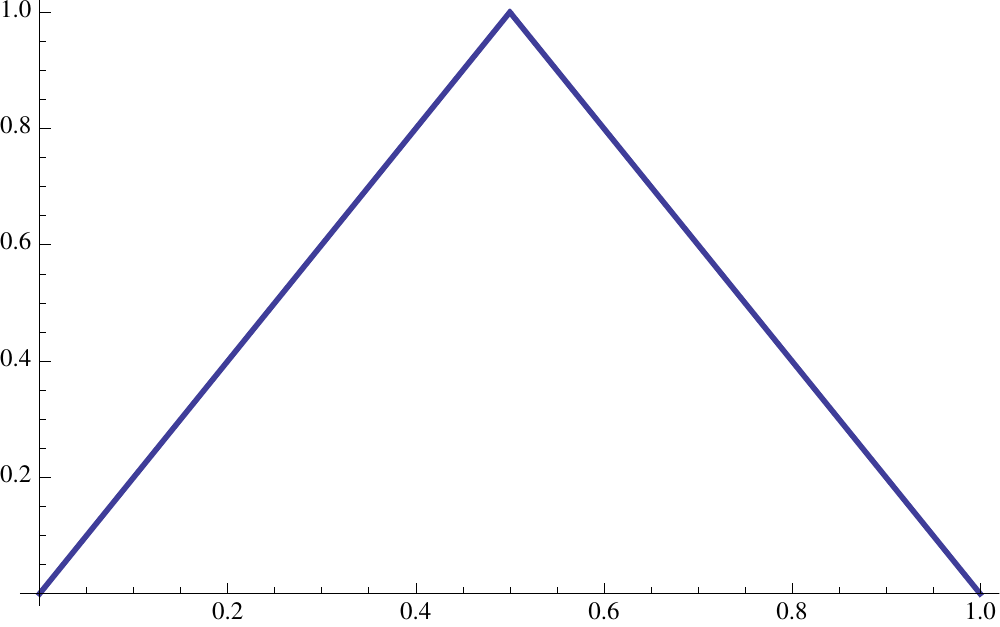}\hspace{.01\linewidth}
    \includegraphics[width=.23\linewidth]{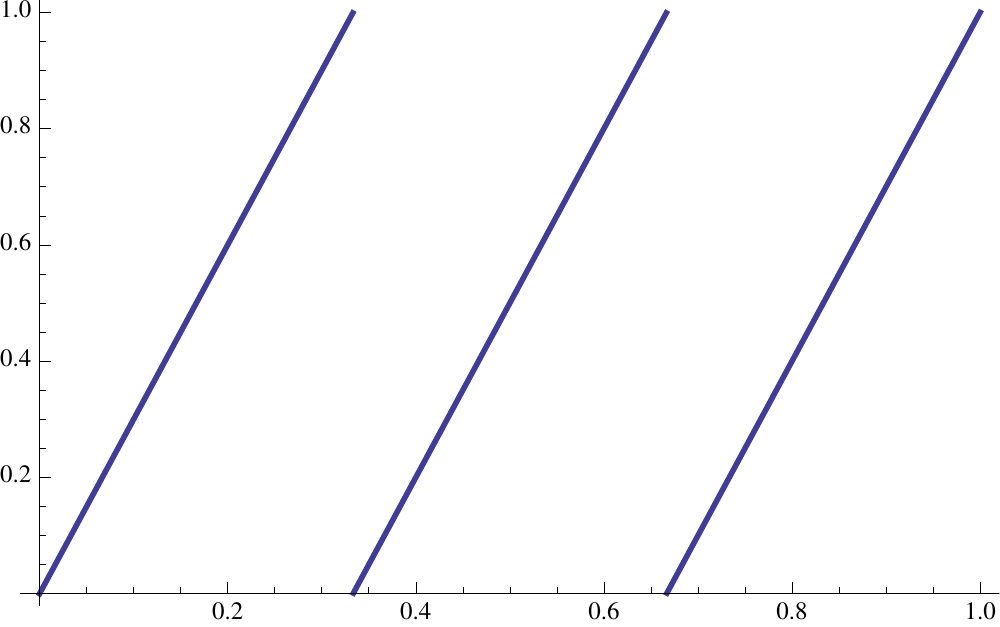}\hspace{.01\linewidth}
    \includegraphics[width=.23 \linewidth]{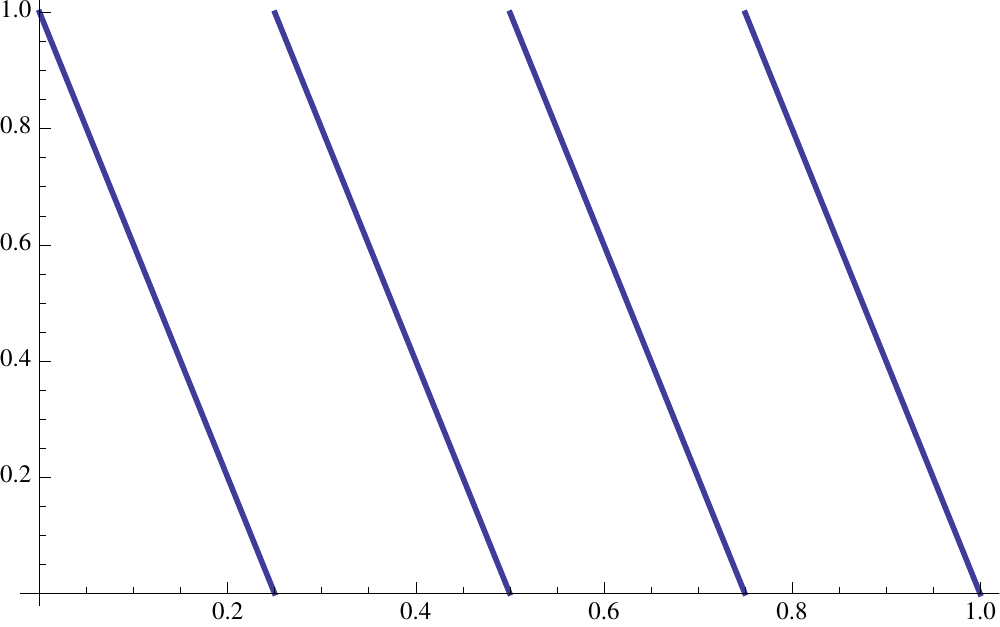}\hspace{.01\linewidth}
    \includegraphics[width=.23\linewidth]{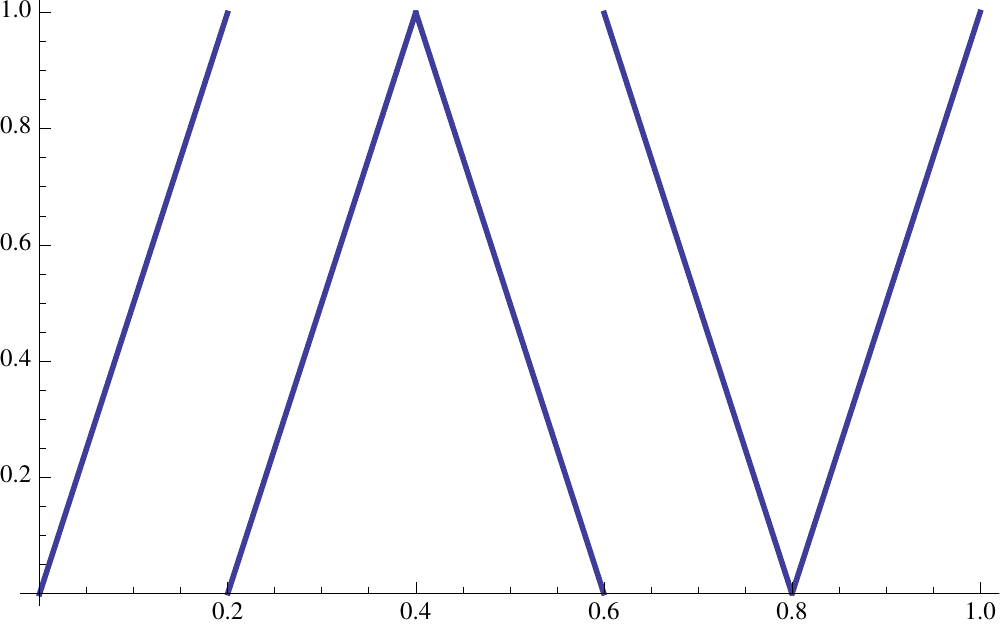}
\caption{The graphs of $M_\sigma$ for $\sigma =+-$, $\sigma =+++$, $\sigma =----$ and $\sigma = + + - - +$, respectively.}
\label{fig:signedshifts}
\end{figure}

We next give another definition of the signed shift that will be used in this paper and which also appeared in \cite{AE}. 
Let $\prec_\sigma$ be the linear order on $\W_k$ defined by $s= s_1s_2s_3\dots\prec_\sigma t_1t_2t_3\dots = t$ if one of the following holds:
\begin{enumerate}
\item $s_1 < t_1$,
\item $s_1 = t_1 \in T^+_\sigma$ and $s_2s_3\ldots\prec_\sigma t_2t_3\ldots,$ or
\item $s_1 = t_1 \in T^-_\sigma$ and $t_2t_3\ldots\prec_\sigma s_2s_3\ldots.$
\end{enumerate}
Equivalently, $s\prec_\sigma t$ if, letting $j\ge1$ be the smallest such that $s_j \neq t_j$, either
$c:=|\{1\le i<j: s_i\in T^-_\sigma\}|$ is even and $s_j<t_j$, or $c$ is odd and $s_j>t_j$.
The signed shift is the map $\Sigma_\sigma:(\W_k,\prec_\sigma)\to (\W_k,\prec_\sigma)$ defined simply by $\Sigma_\sigma(s_1s_2s_3s_4\dots)=s_2s_3s_4\dots$.

For example, let $\sigma = +--$ and consider the periodic point $s = 00110221001\ldots$.Using the ordering defined by $\prec_\sigma$, we can find that $\Pat(s,\Sigma_\sigma,8)=12453786$. 

To show that the two definitions of the signed shift as $\Sigma_\sigma$ and $\Sigma'_\sigma$ are order-isomorphic, consider the order-preserving bijection $\psi_\sigma:(\W_k,\prec_\sigma)\to(\W_k,\lex)$ that maps a word $s=s_1s_2s_3\dots$ to the word $a = a_1a_2a_3\dots$ where
$$a_i = \begin{cases} s_i & \mbox{if }|\{j< i : s_j\in T^-_\sigma\}| \mbox{ is even,} \\ \bar{s_i} & \mbox{if }|\{j< i : s_j\in T^-_\sigma\}| \mbox{ is odd.} \end{cases}$$
It is easy to check that $\psi_\sigma\circ\Sigma_\sigma = \Sigma_\sigma'\circ\psi_\sigma$, and thus $\Allow(\Sigma_\sigma) = \Allow(\Sigma_\sigma')$. 

If $\sigma = +^k$, then $\prec_\sigma$ is the lexicographic order $\lex$, and $\Sigma_\sigma$ is called the {\em $k$-shift} and is typically denoted by $\Sigma_k$.
When $\sigma=-^k$, the map $\Sigma_\sigma$ is called the {\em reverse $k$-shift} and is denoted here by $\rs_k$. When $\sigma=+-$, the map $\Sigma_\sigma$ is the well-known {\em tent map} and is denoted by $\tent$.

\subsection{Finite words and infinite periodic words}

For the proofs, we may need to refer to finite words at time. A finite word $q = q_1q_2\ldots q_n$ is an initial segment of length $n$ of an infinite word. The finite word $q^m$ is the word of length $nm$ which is the concatenation of $q$ with itself $m$ times. For example, if $q = 22010$, then $q^2 = 2201022010$. The infinite word $q^\infty$ is the concatenation of $q$ with itself an infinite number of times.

A finite word $q$ is called \textit{primitive} if there is no way to write $q$ as $q = r^m$ for any strictly shorter finite word $r$ and any $m>1$. If $q = q_1q_2 \ldots q_n$ is primitive, then the infinite word $q^\infty$ is an $n$-periodic word. Notice that under $n$ iterations of $\Sigma_\sigma$, we obtain $\Sigma_\sigma^n(q^\infty) = q^\infty$. An infinite word is called \textit{eventually periodic} if after removing some initial finite segment of the word, we are left with a periodic word. For example, $s = 00010101010101\ldots$ is eventually periodic since $s = 00(01)^\infty$. 

Additionally, for an infinite word $s = s_1s_2\ldots$, the notation $s_{[i,j]}$ for $i<j$ will be used to denote the finite word $s_is_{i+1}\ldots s_j$ and the notation $s_{[i,\infty)}$ will denote the infinite word $s_is_{i+1}\ldots$.


\section{Characterization of allowed patterns of signed shifts} \label{sec:desc allowed}


Some progress toward characterizing the permutations of $\Allow_n(\Sigma_\sigma)$ has been made. In \cite[Theorem 4.1]{Amigosigned}, the author gives necessary conditions (described below) for a permutation to be an allowed pattern of $\Sigma_\sigma$ and posits that these conditions are also sufficient. However, as demonstrated by the counterexamples which follow, the conditions in \cite{Amigosigned} are not enough to guarantee that a permutation is in $\Allow_n(\Sigma_\sigma)$. 

In this section, we will present some convenient notation and state the necessary and sufficient conditions a permutation must satisfy to be an allowed pattern of the signed shift $\Sigma_\sigma$. 

\subsection{Definitions and notation}
For a finite word $q = q_1\cdots q_n$ on letters $\{0,1,2,\ldots, k-1\}$, we denote by $d_i(q) := |\{j : q_j < i\}|$, so that $d_0(q) = 0$, $d_1(q)$ is the number of 0's in $q$, $d_2(q)$ is the number of 0's and 1's, etc.

Let $\S_n^*$ be the set of $*$-permutations, that is permutations where one element has been replaced by a $*$. We denote the elements of $\S_n^*$ by $\pi^*$. We denote by $\C_n^*$ the set of $*$-permutations $\pi^*$ so that the permutation $\pi$ obtains by replacing the $*$ by its missing element is cyclic.
\begin{example}
The element $\pi^* = 36582{*}179 \in \S_9^*$ since the element $4$ in the permutation $\pi =  365824179$ has been replaced with $*$. 
\end{example}
\begin{example}
The element $\pi^* =47861{*}52 \in \S_8^*$ since the element $3$ in the permutation $\pi =  47861352 $ has been replaced with $*$. Also, since $\pi$ is a cyclic permutation, $\pi^* \in \C_8^*$.
\end{example}
Recall that for a given $\pi = \pi_1\cdots \pi_n$, there is a cyclic permutation $\pihat = (\pi_1\ldots \pi_n) = \hp_1\ldots \hp_n$. 
We will define a map \begin{align*}\S_n &\to \C_n^*\\ \pi & \mapsto \pihat^*,\end{align*} where $\pihat^*$is obtained by replacing the entry $\pi_1$ in $\pihat$ with $*$. 

\begin{example}
Let $\pi = 834192675$. Then $\pihat = (834192675) = 964187532$. Since $\pi_1 = 8$, we replace the 8 by a $*$ to get $\pihat^* = 9641{*}7532$. 
\end{example}

Recall that $\tau \in \S^\sigma$ if that there is some {\em $\sigma$-segmentation} of $\pihat$, $0 = e_0\leq \cdots\leq e_k = n$ so that $\tau_{e_t+1} \cdots \tau_{e_t}$ is increasing if $\sigma_t = +$ and decreasing if $\sigma_t = -$. Similarly, we will say that $\tau^*\in\S^{\sigma,*}$ if there is a {\em $*$-$\sigma$-segmentation} of $\tau^*$. The definition in this case is somewhat more complicated since we require extra conditions which depend on the location of the $*$. We define a $*$-$\sigma$-segmentation of $\tau^*$ to be $0 = e_0 \leq e_1\leq \cdots \leq e_k = n$ such that the following properties hold: 
\begin{enumerate}[(a.)]
\item the sequence $\tau^*_{e_t+1} \cdots \tau^*_{e_{t+1}}$ is increasing if $\sigma_t = +$ and decreasing if $\sigma_t = -$.
\item if $\sigma_0 = +$ and $\tau^*_1 \tau^*_2 = *1$, then $e_1 \leq 1$.
\item if $\sigma_{k-1} = +$ and $\tau^*_{n-1} \tau^*_{n} = n*$, then $e_{k-1} \geq n-1$.
\item if $\sigma_0 = \sigma_{k-1} = -$ and both $\tau^*_1 = n$ and $\tau^*_{n-1} \tau^*_{n} = 1*$, then either $e_1=0$ or $e_{k-1} \geq n-1$.
\item if $\sigma_0 = \sigma_{k-1} = -$ and both $\tau^*_n = 1$ and $\tau^*_{1} \tau^*_{2} = *n$, then either $e_{k-1} = n$ or $e_{1} \leq 1$.
\end{enumerate}

\begin{example}\label{ex:cn1}
Suppose $\sigma = +-$. The permutation $\tau^* = 467893{*}21 \in \C_9^*$ has two $*$-$\sigma$-segmentations. Either of $0\leq 4 \leq 9$ or $0\leq 5\leq 9$ satisfy all of the necessary conditions. 
\end{example}
\begin{example}\label{ex:cn3}
Suppose $\sigma = --$. The permutation $\tau^* = 754261{*} \in \C_7^*$ does not have a $*$-$\sigma$-segmentation. In order to satisfy condition (a), it must have the $\sigma$-segmentation $0\leq 4 \leq 7$. However, by condition (d), we must have that $e_1 = 0$, $6$, or $7$, which is not the case.
\end{example}

In \cite{Amigosigned}, the author states the claim (in equivalent terms) that in order for $\pi$ to be an allowed pattern of $\Sigma_\sigma$, it is sufficient for $\tau^* = \hat{\pi}^*$ to satisfy condition (a) above. Here we present two counterexamples, one of which is taken care of by the extra conditions (b)--(e). The other counterexample shows why an extra condition will be necessary for our main theorem.

\begin{example}\label{ex:count1}
For $\sigma = +-$, the permutation $\pi = 591482637$ satisfies the condition (a). However, the word $s$ so that $\pi = \Pat(s_1s_2\cdots, \Sigma_\sigma, 9)$ is forced (by the conditions of the theorem) to starts with either $$s_1s_2\cdots s_9 = 010010101$$ or $$s_1s_2\cdots s_9 = 110010101.$$ Extending this to an infinite word, we are forced to have $$s_1s_2\cdots = 010010101(01)^\infty \mbox { or } 110010101(01)^\infty .$$ But the patterns of length 9 for these words are undefined. Indeed, $\pi$ is not an allowed pattern of the tent map.
\end{example}
\begin{example}\label{ex:count3}
For $\sigma = --$, the permutation $\pi = 3425617$ satisfies condition (a). This forces $s_1\cdots s_n = 0001101$. However, now matter what you assign $s_{[7,\infty)}$, we have that $s_{[5,\infty)}$ must be larger. If you assign $s_{[7,\infty)} = (10)^\infty$, which is the largest word possible with respect to the ordering $\prec_\sigma$, then the pattern for $s$ is undefined.
\end{example}

Notice that the two examples \ref{ex:cn1} and \ref{ex:cn3} satisfy $\tau^* = \pihat^*$ for the permutations in the counterexamples presented in Examples \ref{ex:count1} and \ref{ex:count3}, respectively. The extra conditions presented in the definition of a $*$-$\sigma$-segmentation take care of Example  \ref{ex:count3}, but not Example  \ref{ex:count1}. There is an additional extra condition required in order for a given permutation to be an allowed pattern of $\Sigma_\sigma$. This is exactly the condition $(\dagger)$ given in the statement of the theorem below, which completely characterizes the permutations of $\Allow_n(\Sigma_\sigma)$.

\begin{theorem}\label{thm:desc allowed}
A permutation $\pi \in \Allow_n(\Sigma_\sigma)$ if and only if $\hat\pi^* \in \C^{\sigma, *}$ and also $\pi$ satisfies the following condition:
\vspace{.3cm}

$(\dagger)$ There exists some $*$-$\sigma$-segmentation $0 = e_0\leq e_1\leq \cdots \leq e_k = n$ of $\hat\pi^*$ so that there is no $b$ satisfying 
	\begin{itemize}
	\item $\st(\pi_{n-2b}\pi_{n-b}\pi_n)= 312$ or $132$, and
	\item $e_t< \pi_{n-b-i}\leq e_{t+1}$ if and only if $e_t<\pi_{n-i}\leq e_{t+1}$ for all $1\leq i \leq b$.
	\end{itemize}
\end{theorem}

\begin{example}
Let us see why $(\dagger)$ takes care of the counterexample presented in Example \ref{ex:count1}. As before, $\sigma = +-$ and $\pi = 591482637$. The permutation $\pihat^* = 467893{*}21$ has two $*$-$\sigma$-segmentations: $0\leq 4 \leq 9$ or $0\leq 5\leq 9$. Consider $b = 2$. We claim that this choice of $b$ will violate $(\dagger)$ for either $*$-$\sigma$-segmentation. Indeed, if $b = 2$, then $\st(\pi_{n-2b}\pi_{n-b}\pi_n) = \st(\pi_5\pi_7\pi_9) = \st(867)= 312$ and also $4,5< \pi_5, \pi_7 = 8,6 \leq 9$ and $0<\pi_6, \pi_8 = 2,3 \leq 4,5$. 
Therefore, since $\pi$ violates $(\dagger)$, $\pi \not\in \Allow_n(\Sigma_\sigma)$. 
\end{example}

To prove this theorem, we require several lemmas. For a word segment $q = q_1\cdots q_m$ we define $\nn{q} = |\{ i : \sigma_{q_i} = - \}|.$

The first two lemmas will take care of the forward direction, that if $\pi \in \PP_n(\Sigma_\sigma)$, then $\hat\pi^* \in \C^{\sigma, *}$ and $(\dagger)$ holds.  
\begin{lemma}\label{lem:pihat in C}
If $\pi \in \Allow_n(\Sigma_\sigma)$, then $\pihat^*\in \C^{\sigma,*}$.
\end{lemma}
\begin{proof}
Suppose $\pi = \Pat(s, \Sigma_\sigma, n)$ where $s = s_1s_2\cdots$ and suppose $d_t = d_t(s_1\cdots s_n)$ for all $0\leq t \leq k$. Then we claim that $e_t = d_t$ for all $0\leq t \leq k$ is a $*$-$\sigma$-segmentation of $\pihat^*$. 

Observe that if $e_t < \pi_i<\pi_j \leq e_{t+1}$, then $s_{[i,\infty)} \prec_\sigma s_{[j,\infty)}$ where $s_i = s_j = t$. If $\sigma = +$, then $s_{[i,\infty)} \prec_\sigma s_{[j,\infty)}$ implies that $s_{[i+1,\infty)} \prec_\sigma s_{[j+1,\infty)}$ so if $i,j<n$, this implies that $\pi_{i+1}<\pi_{j+1}$, which in turn implies that $\pihat^*_i <\pihat^*_j$. Therefore the sequence $\pihat^*_{e_t+1} \cdots \pihat^*_{e_{t+1}}$ is increasing when $\sigma = +$. Similarly if $\sigma = -$, then $s_{[i,\infty)} \prec_\sigma s_{[j,\infty)}$ implies that $s_{[i+1,\infty)} \succ_\sigma s_{[j+1,\infty)}$ so if $i,j<n$, this implies that $\pi_{i+1}>\pi_{j+1}$, which in turn implies that $\pihat^*_i >\pihat^*_j$. Therefore the sequence $\pihat^*_{e_t+1} \cdots \pihat^*_{e_{t+1}}$ is decreasing when $\sigma = -$. Therefore, $\hp$ satisfies condition (a). 

If $\sigma_0 = +$ and $\pihat^*_1\pihat^*_2 = *1$, this must mean that $\pi_{n-1}\pi_n = 21$. If $e_1>1$, then we would have $s_{n-1} = s_n = 0$ and that $s_{[n,\infty)} \prec_\sigma s_{[n-1,\infty)}$. It is clear that this is impossible. Since $\Pat(s, \Sigma_\sigma, n)$ is defined, there is some $m>n$ so that $s_m >0$ and $s_j = 0$ for all $n-1\leq j <m$, but then since $\sigma_0 = +$, this would imply that $s_{[m,\infty)} \prec_\sigma s_{[m-1,\infty)}$, which is a contradiction since $s_m >0$ and $s_{m-1} = 0$. 
Similarly, for when if $\sigma_{k-1} = +$ and $\pihat^*_{n-1}\pihat^*_n = n*$. Therefore, $\hp$ satisfies conditions (b) and (c). 

If $\sigma_0 = \sigma_{k-1} = -$ and both $\pihat^*_1 = n$ and $\pihat^*_{n-1} \pihat^*_{n} = 1*$, then $\pi_{n-2}\pi_{n-1}\pi_n  = (n-1)1n$. If both $e_1>0$ and $e_{k-1}<n-1$, then this implies that $s_{n-2}s_{n-1}s_n = (k-1)0(k-1)$ and $s_{[n-2,\infty)}\prec_\sigma s_{[n,\infty)}$. Let $m>n$ denote the first place for $i>0$ where either $s_{n+2i}<k-1$ or $s_{n+2i -1}>0$. Notice that since $\sigma_0 = \sigma_{k-1} = -$, $s_{[n-2,\infty)}\prec_\sigma s_{[n,\infty)}$ implies that $s_{[n-1,\infty)}\succ_\sigma s_{[n+1,\infty)}$ and the signs continue to alternate until we reach $s_m$. If $m = n+2i$ and $s_{m}<k-1$, then since $s_m$ is an even number of steps from $s_n$, we alternate signs an even number of times to get that $s_{[m-2,\infty)}\prec_\sigma s_{[m,\infty)}$ which is a contradiction since $s_{m-2} = k-1$ and $s_{m}< k-1$. For similar reason, if $m = n+2i-1$ and $s_m>0$, we arrive at a contradiction. Therefore, we must have that either $e_1 = 0$ or $e_{k-1}\geq n-1$. 
The case is similar when $\sigma_0 = \sigma_{k-1} = -$ and both $\pihat^*_n = 1$ and $\pihat^*_{1} \pihat^*_{2} = *n$. Therefore, $\hp$ satisfies conditions (d) and (e) and the lemma is proven.
\end{proof}

\begin{lemma}\label{lem:pi dagger}
If $\pi \in \Allow_n(\Sigma_\sigma)$, then $\pi$ satisfies $(\dagger)$.
\end{lemma}
\begin{proof}
Suppose again that $\pi = \Pat(s, \Sigma_\sigma, n)$ where $s = s_1s_2\cdots$.
If there were no such $\sigma$-segmentation satisfying $(\dagger)$, then no matter which word $s$ we choose, we have $s_{n-2b}\cdots s_{n-1} = q^2$ for some $b$. Finding $s_{n+1}s_{n+2}\cdots$ becomes impossible since $s_ns_{n+1}s_{n+2}\cdots$ lies between $s_{n-2b}\cdots  = qqs_n\cdots $ and $s_{n-b}\cdots = qs_n\cdots $ which forces $s_n\cdots s_{n+b} = q$, and so on. We end up with $s_{n-2b}\cdots  = q^\infty$ and so $\Pat(s,\Sigma_\sigma, n)$ is undefined.
\end{proof}

The next lemmas will allow us to prove the reverse direction of Theorem \ref{thm:desc allowed}.
The idea behind the proofs of these lemmas will be that for some $\pi$ satisfying the conditions of Theorem \ref{thm:desc allowed}, we construct a word $s$ so that $\pi$ is the pattern of $s$ with respect to $\Sigma_\sigma$, proving that $\pi$ is an allowed pattern of $\Sigma_\sigma$.
In these lemmas, we consider the cases when $\pi_n =1$, $\pi_n = n$, and $1<\pi_n<n$. 

Let $W_\sigma$ and $w_\sigma$ be the greatest and least words, respectively, associated to ordering $\prec_\sigma$. That is, we define $W_\sigma$ and $w_\sigma$ by
$$W_\sigma = \begin{cases} (k-1)^\infty & \sigma_{k-1} = + \\ (k-1)0^\infty & \sigma_{k-1} = - \, \mbox{ and } \, \sigma_0 = + \\ ((k-1)0)^\infty & \sigma_{k-1} = - \, \mbox{ and } \, \sigma_0 = - \end{cases} $$
and 
$$w_\sigma = \begin{cases} 0^\infty & \sigma_{0} = + \\ 0(k-1)^\infty & \sigma_{0} = - \, \mbox{ and } \, \sigma_{k-1} = + \\ (0(k-1))^\infty & \sigma_{0} = - \, \mbox{ and } \, \sigma_{k-1} = -. \end{cases} $$

\begin{lemma}\label{lem:pin=1}
If $\pi_n = 1$, $\pihat^* \in \C^{\sigma,*}$, and $\pi$ satisfies $(\dagger)$, then $\pi \in \Allow_n(\Sigma_\sigma)$. 
\end{lemma}
\begin{proof}
First take a $*$-$\sigma$-segmentation of $\pi$ specified by ($\dagger$). Let $s_1\cdots s_n$ be the $\pi$-monotone word associated to this $*$-$\sigma$-segmentation. 
Let $s := s_1\cdots s_{n-1}w_\sigma$. 

First, we show that $\Pat(s, \Sigma_\sigma, n)$ exists. When $\sigma_0 = +$, we have that $w_\sigma = 0^\infty$. Therefore, the pattern of $s$ will not be defined only if $s_{n-1} = 0$. If this is the case, then we must have that $\pi_{n-1} = 2$.  If it were not, then there would be some $1\leq j <n-1$ so that $\pi_j = 2$ and $\pi_{n-1}>2$. But then since $e_0<\pi_j<\pi_{n-1}\leq e_1$, we must have that $\pi_{j+1}<\pi_n = 1$, which is impossible. Therefore, $\pi_{n-1}\pi_n = 21$, which implies that $\pihat^*_1 \pihat^*_2 = *1$. However, since $0=e_0\leq e_1\leq \cdots \leq e_k = n$ is a $*$-$\sigma$-segmentation, we must have $e_1 \leq 1$. Since $s_1\cdots s_n$ is a $\pi$-monotone word, this implies that $s_n =0$ and $s_{n-1} \neq 0$, and we get a contradiction. Thus, the pattern is defined in this case.

When $\sigma_0 = -$ and $\sigma_{k-1} = +$, we have $w_\sigma = 0(k-1)^\infty$, and so the pattern must be defined in the case.

In the case that $\sigma_0 =\sigma_{k-1} = -$, we have $w_\sigma = (0(k-1))^\infty$, and so the pattern of $s$ will not be defined only if $s_{n-2}s_{n-1} = 0(k-1)$. Suppose this is the case. Then similarly to above, we can show that $\pi_{n-2}\pi_{n-1}\pi_n = 2n1$. This would mean that $\pihat^*_1 = n$ and $\pihat^*_{n-1} \pihat^*_{n} = 1*$, and so either $e_1=0$ or $e_{k-1} \geq n-1$. Therefore either $\pi_1 \neq 0$ or $\pi_{n-1} \neq k-1$, which gives us a contradiction. Therefore, the pattern is defined. 

It remains to show that $\pi_i<\pi_j$ implies $s_{[i,\infty)}\prec_\sigma s_{[j,\infty)}$. 
If there is some $t$ so that $\pi_i\leq e_t<\pi_j$, then we are done since this would imply $s_i<s_j$. 
Otherwise, there is some $t$ so that $e_t \leq \pi_i<\pi_j<e_{t+1}$ and thus $s_i = s_j = t$. 
If $\pi_i<\pi_j$, let $m$ be so that $s_i\cdots s_{i+m-1} = s_j\cdots s_{j+m-1}$ and $s_{i+m} \neq s_{j+m}$. First, assume that $i+m, j+m \leq n$. Since for all $1\leq \ell <m$, there is always some $t$ so that $e_t \leq \pi_{i+\ell}, \pi_{j+\ell}<e_{t+1}$, by definition of $\S^\sigma$, when $t \in T^+$, we have $\pi_{i+\ell} <\pi_{j+\ell}$ if and only if $\pi_{i+\ell+1} <\pi_{j+\ell+1}$ and when $t \in T^-$, we have $\pi_{i+\ell} <\pi_{j+\ell}$ if and only if $\pi_{i+\ell+1} >\pi_{j+\ell+1}$.

Suppose $\nn{s_i\cdots s_{i+m-1}}$ is even. 
Then we have that $\pi_{i}<\pi_{j}$ if and only if $\pi_{i+m}<\pi_{j+m}$ since the inequality changes and even number of times. Since $s_{i+m} \neq s_{j+m}$, this implies that $s_{i+m} <s_{j+m}$, and thus $s_{[i+m,\infty)}\prec_\sigma s_{[j+m,\infty)}$. However, we have $s_{[i,\infty)}\prec_\sigma s_{[j,\infty)}$ if and only if $s_{[i+m,\infty)}\prec_\sigma s_{[j+m,\infty)}$ since we have the inequality $\prec_\sigma$ switch an even number of times.
 Similarly, if $\nn{s_i\cdots s_{i+m-1}}$ is odd, $\pi_i<\pi_j$ if and only if $\pi_{i+m}>\pi_{j+m}$, which implies that $s_{i+m}>s_{j+m}$. This in turn implies that $s_{[i+m,\infty)}\succ_\sigma s_{[j+m,\infty)}$, which happens if and only if $s_{[i,\infty)}\prec_\sigma s_{[j,\infty)}$.

If one of $i+m$ or $j+m$ is greater than $n$, then take $m'$ to be such that either $i+m'$ or $j+m'$ is $n$ and the other is less than $n$. Notice that $\pi_n = 1< \pi_\ell$ for all $\ell<n$ and $s_{[n,\infty)} = w_\sigma \prec_\sigma s_{[\ell,\infty)}$ for all $\ell<n$. 
The same argument as above applies in this case where we now check whether $\nn{s_{i}\cdots s_{i+m'}}$ is even or odd.
It follows that $\pi_i<\pi_j$ implies $s_{[i,\infty)}\prec_\sigma s_{[j,\infty)}$. 
\end{proof}
\begin{lemma}\label{lem:pin=n}
If $\pi_n = n$, $\pihat^* \in \C^{\sigma,*}$, and $\pi$ satisfies $(\dagger)$, then $\pi \in \Allow_n(\Sigma_\sigma)$. 
\end{lemma}
\begin{proof}
Let $s = s_1\cdots s_{n-1}W_\sigma$. The argument that the pattern for this word exists and is $\pi$ is parallel to the one above.
\end{proof}

The case when $1\leq \pi_n\leq n$ is a little more complicated.  
Again, take a $*$-$\sigma$-segmentation of $\pi$ specified by ($\dagger$). Let $s_1\cdots s_n$ be the $\pi$-monotone word associated to this $*$-$\sigma$-segmentation.
Define two words $s^{(1)}$ and $s^{(2)}$ as follows.
$$ s^{(1)} = \begin{cases}up^{n-1}w_\sigma & \mbox{ $n$ is even or $\nn{p}$ even} \\ up^{n-1}W_\sigma & \mbox{ $n$ is odd and $\nn{p}$ odd}  \end{cases}$$
where $\pi_x = \pi_n +1$, $u = s_1\cdots s_{x-1}$ and $p = p^{(1)} = s_x\cdots s_{n-1}$
and 
$$ s^{(2)} = \begin{cases} up^{n-1}W_\sigma & \mbox{ $n$ is even or $\nn{p}$ even} \\ up^{n-1}w_\sigma & \mbox{ $n$ is odd and $\nn{p}$ odd}\end{cases} $$
where $\pi_y = \pi_n -1$, $u = s_1\cdots s_{y-1}$ and $p = p^{(2)}= s_y\cdots s_{n-1}$. We will show that for at least one of these words, the pattern of the word is $\pi$.

\begin{lemma}\label{lem:p primitive}
As defined above, $p$ must be either primitive or $p = q^2$ where $q$ is primitive and $\nn{q}$ is odd.
\end{lemma}
\begin{proof}
We will prove this for $s^{(1)}$. The proof for $s^{(2)}$ is similar. Assume $p = q^r$ with $|q| = b$. Then if $\nn{q}$ is even and $\pi_x<\pi_{x+b}$, then we must have $\pi_x<\pi_{x+b}<\pi_{x+2b}<\cdots<\pi_{x+rb}= \pi_n$. Since $\pi_x = \pi_n+1$, we have a contradiction. Alternatively, if $\nn{q}$ is even and $\pi_x>\pi_{j+b}$, then we must have $\pi_x>\pi_{x+b}>\pi_{x+2b}>\cdots>\pi_{x+rb}= \pi_n$. Since $\pi_x = \pi_n-1$, we must have $r =1$. 

If $\nn{q}$ is odd and $r>2$ is even, then we can reduce it to the previous case for $q' = qq$. It remains to show that if $\nn{q}$ is odd and $r>2$ is odd, we get a contradiction. There are 4 possibilities:
\begin{itemize}
\item If $\pi_x<\pi_{x+b}$ and $\pi_{x}<\pi_{x+2b}$, then $\pi_x<\pi_{x+b}>\pi_{x+2b}<\cdots >\pi_{n-b}<\pi_n$ (so $\pi_{n-b}<\pi_n$). Also, $\pi_x<\pi_{x+2b}<\pi_{x+4b}<\cdots< \pi_{n-b}$. But $\pi_{n-b}<\pi_n$. So we have $\pi_x<\pi_{x+2b}<\pi_{x+4b}<\cdots< \pi_{n-b}<\pi_n$, which contradicts that $\pi_x= \pi_n+1$.
\item If $\pi_x>\pi_{x+b}$ and $\pi_{x}<\pi_{x+2b}$, then $\pi_x<\pi_{x+2b}<\pi_{x+4b}<\cdots< \pi_{n-b}$. That implies that $\pi_{x+b}>\pi_{x+3b}>\pi_{x+5b}>\cdots> \pi_{n-2b}>\pi_n$. But $\pi_x>\pi_{x+b}$, so $\pi_x>\pi_{x+b}>\pi_{x+3b}>\pi_{x+5b}>\cdots> \pi_{n-2b}>\pi_n$, but $\pi_x =\pi_n+1$, so this implies that $r=1$.
\item If $\pi_x<\pi_{x+b}$ and $\pi_{x}>\pi_{x+2b}$, then $\pi_x>\pi_{x+2b}>\pi_{x+4b}>\cdots>\pi_{n-b}$. That implies that $\pi_{x+b}<\pi_{x+3b}<\pi_{x+5b}<\cdots< \pi_{n-2b}<\pi_n$. But $\pi_x<\pi_{x+b}$, so $\pi_x<\pi_{x+b}<\pi_{x+3b}<\pi_{x+5b}<\cdots<\pi_{n-2b}<\pi_n$, which contradicts that $\pi_x =\pi_n+1$.
\item If $\pi_x>\pi_{x+b}$ and $\pi_{x}>\pi_{x+2b}$, then $\pi_x>\pi_{x+b}<\pi_{x+2b}>\cdots <\pi_{n-b}>\pi_n$ (so $\pi_{n-b}>\pi_n$). Also, $\pi_x>\pi_{x+2b}>\pi_{x+4b}>\cdots> \pi_{n-b}$. But $\pi_{n-b}>\pi_n$. So we have $\pi_x>\pi_{x+2b}>\pi_{x+4b}>\cdots>\pi_{n-b}>\pi_n$, but $\pi_x =\pi_n+1$, so this implies that $r=1$.
\end{itemize}
Therefore we either have $p$ is primitive or $p=q^2$ where $q$ is primitive and $\nn{q}$ is odd. 
\end{proof}

\begin{lemma}\label{lem:s exists}
The pattern of $s$ exists for at least one of $s = s^{(1)}$ or $s = s^{(2)}$.
\end{lemma}
\begin{proof}
We prove this by cases. 

{\it Case 1.} Suppose $\sigma_0 = \sigma_{k-1} =+$. Then $w_\sigma = 0^\infty$ and $W_\sigma = (k-1)^\infty$. The only way the pattern could be undefined would be if $p^{(1)} = 0$ or if $p^{(2)}=k-1$. Notice that in both cases, $\nn{p^{(i)}}=0$ is even. 
If $p^{(1)} = 0$, the pattern of $s^{(1)}$ cannot exist since in this case we would have $s^{(1)}_{[n-1,\infty)} = 0^\infty$, and so $s^{(1)}_{[n-1,\infty)}$ and $s^{(1)}_{[n,\infty)}$ would be incomparable. However, in this case, the pattern of $s^{(2)}$ must be defined since $p^{(2)} \neq k-1$ since $s_{n-1} = 0$ and so $p^{(2)}$ must end in 0. Similarly, if $p^{(2)} = k-1$, then the patterns of $s^{(2)}$ does not exist, but the pattern of $s^{(1)}$ does.

{\it Case 2.} Suppose $\sigma_0=+$ and $\sigma_{k-1} = -$. Then $w_\sigma = 0^\infty$ and $W_\sigma = (k-1)0^\infty$. The only way the pattern could be undefined in this case is if $p^{(1)} = 0$ (since $W_\sigma$ is not periodic). For the same reasons as in Case 1, the pattern of $s^{(1)}$ will not be defined, but the pattern of $s^{(2)}$ will be.  

{\it Case 3.} Suppose $\sigma_0=-$ and $\sigma_{k-1} = +$. Then $w_\sigma = 0(k-1)^\infty$ and $W_\sigma = (k-1)^\infty$. The only way the pattern could be undefined in this case is if $p^{(2)} = k-1$ (since $w_\sigma$ is not periodic). For the same reasons as in Case 1, the pattern of $s^{(2)}$ will not be defined, but the pattern of $s^{(1)}$ will be.  

{\it Case 4.} Suppose $\sigma_0=\sigma_{k-1} = -$. Then $w_\sigma = (0(k-1))^\infty$ and $W_\sigma = ((k-1)0)^\infty$. The only way the pattern could be undefined in this case is if $p^{(1)} = 0(k-1)$ or if $p^{(2)} = (k-1)0$. In both cases $\nn{p^{(i)}}$ is even. The proof is the same as in Case 1. 
\end{proof}

In the remaining lemmas, we show that if the pattern of $s = s^{(1)}$ exists, then the pattern is $\pi$. The proof for $s^{(2)}$ is very similar. For the next several lemmas, assume that $s = s^{(1)}$ and that $x,p,u,$ and $n$ are all defined as in the definition of $s^{(1)}$.

\begin{lemma}\label{lem:n<x}
$s_{[n,\infty)}\prec_\sigma s_{[x,\infty)}$
\end{lemma} 
\begin{proof}
By definition of $s= s^{(1)}$, this is equivalent to $p^{n-2}w_\sigma \prec_\sigma p^{n-1}w_\sigma$ when $n$ is even or $\nn{p}$ is even and $p^{n-2}W_\sigma \prec_\sigma p^{n-1}W_\sigma$ when both $n$ and $\nn{p}$ are odd. In the first case when $n$ is even or if $\nn{p}$ is even, then $\nn{p^{n-2}}$ is even and so $p^{n-2}w_\sigma \prec_\sigma p^{n-1}w_\sigma$ if and only if $w_\sigma \prec_\sigma pw_\sigma$, which is trivially true since $w_\sigma$ is the smallest word with respect to $\prec_\sigma$. Similarly, if $n$ is odd and $\nn{p}$ is odd, then $\nn{p^{n-2}}$ is odd and so $p^{n-2}W_\sigma \prec_\sigma p^{n-1}W_\sigma$ if and only if $W_\sigma \succ_\sigma pW_\sigma$, which is also trivially true since $W_\sigma$ is the largest word with respect to $\prec_\sigma$.
\end{proof}

\begin{lemma}\label{lem:no c}
There is no $1\leq c\leq n$ such that $s_{[n,\infty)}\prec_\sigma s_{[c,\infty)}\prec_\sigma s_{[x,\infty)}$. 
\end{lemma}
\begin{proof} We prove this by cases. 

{\it Case 1.} Assume $p$ is primitive and either $n$ is even or $\nn{p}$ is even. If there were such a $c$, then we would have $p^{n-2}w_\sigma\prec_\sigma s_{[c,\infty)}\prec_\sigma p^{n-1}w_\sigma$. Therefore, $s_{[c,\infty)} = p^{n-2}v$ for some word $v$ such that $w_\sigma \prec_\sigma v \prec_\sigma pw_\sigma$ (since either $n$ is even or $\nn{p}$ is even). Since $p$ is primitive, there can be no overlap of the first $p$ in $s_{[c,\infty)}$ with both the first and second occurrence of $p$ in $s_{[x,\infty)}$. Since $c \neq x,n$, then we must have $c<x$. 

If some of the occurrences of $p$ in $s_{[c,\infty)}$ overlap with those in $s_{[x,\infty)}$, then $v$ must start with $p$. If $\nn{p}$ is even, this contradicts that $v \prec_\sigma pw_\sigma$ since $pw_\sigma$ is the smallest word starting with $p$. 
If $\nn{p}$ is odd, then we would have that $u$ ends in $p$. Therefore $s = up^{n-1}w_\sigma = u'p^nw_\sigma$ for some $u'$. Suppose $z = n - 2(n-x)$. Then $s_{[z,\infty)} = p^nw_\sigma$. Recall that $s_1\ldots s_{n-1}$ were chosen to be the $\pi$-monotone word obtained from a $\sigma$-segmentation of $\pihat^*$ satisfying $(\dagger)$. Therefore, since $\pi_x>\pi_n$, we cannot have that $\pi_n>\pi_z$. Since $\pi_x = \pi_n +1$, then we must have that $\pi_n<\pi_x<\pi_z$. We then have that $\pi_x<\pi_z$, $\nn{p}$ is odd, and for all $1\leq i \leq n-x$, $e_t< \pi_{x+i}\leq e_{t+1}$ if and only if $e_t<\pi_{z+i}\leq e_{t+1}$ for some $t$. Together, these imply that $\pi_n>\pi_x$, which is a contradiction. 

If there are no occurrences of $p$ in $s_{[c,\infty)}$ which overlap with those in $s_{[x,\infty)}$, then we must have that $c = 1$, $u = p^{n-2}$ and $v = s_{[x,\infty)}$. In this case, we still have the same issues as above. 

{\it Case 2.} Assume $p$ is primitive, $n$ is odd and $\nn{p}$ is odd. The proof is very similar to the situation in Case 1 above when $\nn{p}$ is odd.

{\it Case 3.} Assume $p$ is not primitive. Then $p=q^2$ where $q$ is primitive and $\nn{q}$ is odd. Here, $\nn{p}$ is always even. We would have that $q^{2n-4}w_\sigma\prec_\sigma s_{[c,\infty)}\prec_\sigma q^{2n-2}w_\sigma$ and so $s_{[c,\infty)} = q^{2n-4}v$ for some word $v$ such that $w_\sigma \prec_\sigma v \prec_\sigma q^2w_\sigma$. If we had that $c>x$, then there is exactly one place it could be. But then $v = qw_\sigma$, the largest word that starts with $q$. This violates the fact that $v \prec_\sigma q^2w_\sigma$. Therefore $c<x$. 

Since $q^2w_\sigma$ is the smallest word that starts with $q^2$, there can be no overlap between the occurrences of $q$ in $s_{[c,\infty)}$ and $s_{[x,\infty)}$ since otherwise, $v$ would start with $q^2$. Additionally, since $q^{2n-4}$ has size at least $2(n-2)>n-2$ and so would be longer than all of $u$. 
\end{proof}

\begin{lemma}\label{lem:pat is pi}
$\Pat(s, \Sigma_\sigma, n) = \pi$
\end{lemma}
\begin{proof}
To do this we show that $\pi_i<\pi_j$ implies $s_{[i,\infty)}\prec_\sigma s_{[j,\infty)}$. 
In Lemma \ref{lem:n<x}, we have shown this is true for the case when $i = n$ and $j = x$.
%

Suppose $\pi_i<\pi_j$. If there is some $t$ so that $\pi_i\leq e_t<\pi_j$, then we are done since this implies that $s_i<s_j$ since $s_1s_2\ldots s_n$ were chosen to be $\pi$-monotone with respect to the $*$-$\sigma$-segmentation $0=e_0\leq \cdots \leq e_k = n$.
Otherwise, there is some $t$ so that $e_t \leq \pi_i<\pi_j<e_{t+1}$ and thus $s_i = s_j = t$. 
If $i = n$, then since $\pi_x = \pi_n+1$, $\pi_i<\pi_j$ is equivalent to $\pi_x<\pi_j$ (assuming $j \neq x$, since we already have proven this case). Similarly, if $j = n$, $\pi_i<\pi_j$ is equivalent to $\pi_i<\pi_x$. Therefore it is enough to show that $\pi_i<\pi_j$ implies $s_{[i,\infty)}\prec_\sigma s_{[j,\infty)}$ for $i, j <n$. 

If $\pi_i<\pi_j$, let $m$ be so that $s_i\cdots s_{i+m-1} = s_j\cdots s_{j+m-1}$ and $s_{i+m} \neq s_{j+m}$. First, assume that $i+m, j+m \leq n$. Since for all $1\leq \ell <m$, there is always some $t$ so that $e_t \leq \pi_{i+\ell}, \pi_{j+\ell}<e_{t+1}$, by definition of $\S^\sigma$, when $t \in T^+$, we have $\pi_{i+\ell} <\pi_{j+\ell}$ if and only if $\pi_{i+\ell+1} <\pi_{j+\ell+1}$ and when $t \in T^-$, we have $\pi_{i+\ell} <\pi_{j+\ell}$ if and only if $\pi_{i+\ell+1} >\pi_{j+\ell+1}$.

Suppose $\nn{s_i\cdots s_{i+m-1}}$ is even. 
Then we have that $\pi_{i}<\pi_{j}$ if and only if $\pi_{i+m}<\pi_{j+m}$ since the inequality changes and even number of times. Since $s_{i+m} \neq s_{j+m}$, this implies that $s_{i+m} <s_{j+m}$, and thus $s_{[i+m,\infty)}\prec_\sigma s_{[j+m,\infty)}$. However, we have $s_{[i,\infty)}\prec_\sigma s_{[j,\infty)}$ if and only if $s_{[i+m,\infty)}\prec_\sigma s_{[j+m,\infty)}$ since we have the inequality $\prec_\sigma$ switch an even number of times.

 Similarly, if $\nn{s_i\cdots s_{i+m-1}}$ is odd, $\pi_i<\pi_j$ if and only if $\pi_{i+m}>\pi_{j+m}$, which implies that $s_{i+m}>s_{j+m}$. This in turn implies that $s_{[i+m,\infty)}\succ_\sigma s_{[j+m,\infty)}$, which happens if and only if $s_{[i,\infty)}\prec_\sigma s_{[j,\infty)}$.
If we do not have that $i+m, j+m \leq n$, then whenever we ever reach $n$, we can return back to $s_{[x,\infty)}$ and $\pi_x$ by the second paragraph of this proof.
\end{proof}

\begin{proof}[Proof of Theorem \ref{thm:desc allowed}.]
The forward direction of the proof follows directly from Lemmas \ref{lem:pihat in C} and \ref{lem:pi dagger}. The reverse direction follows from Lemma \ref{lem:pin=1}, when $\pi_n = 1$, from Lemma \ref{lem:pin=n} when $\pi_n = n$, and from Lemmas \ref{lem:p primitive}, \ref{lem:s exists}, \ref{lem:n<x}, \ref{lem:no c}, and \ref{lem:pat is pi} when $1<\pi_n<n$ since we show that there does exist some word so that when $\pi$ satisfies the conditions of the theorem, the pattern of the word is in fact $\pi$. 
\end{proof}

\section{Enumerating allowed patterns of signed shifts}\label{sec:enum allowed}

\subsection{Bounds for the general signed shift}

In this section, we provide an upper bound on the size of the set $\Allow_n(\Sigma_\sigma)$ of allowed patterns for the signed shift. Let $$a(n,k) = \sum_{t=1}^{n-1} \psi_k(t) k^{n-t-1}$$ where $\psi_k(t)$ is the number of primitive words on $k$ letters of length $t$. Notice that $a(n,k)$ counts the number of ways to write a word of length $n-1$ on $k$ letters as the $up$ where $|u| = n-t-1$, $|p| = t$, and $p$ is primitive. We will define an {\em allowed interval for length $n$} as an open interval where the endpoints are eventually periodic points that are of the form $s_1\cdots s_t(s_{t+1} \cdots s_{n-1})^\infty$,  $s_1\cdots s_{n-1} w_\sigma$, or $s_1\cdots s_{n-1}W_\sigma$. The allowed intervals are exactly the connected components of the sets 
$$A^\sigma_\pi = \{s \in \W_k : \Pat(s,\Sigma_\sigma, n) = \pi\}.$$ 
This means that for each allowed interval for length $n$, the pattern of length $n$ realized by each point in the interval is the same.
These allowed intervals partition the domain and do not contain any endpoints of other allowed intervals. Notice that $a(n,k)$ is the number of endpoints of the form $s_1\cdots s_t(s_{t+1} \cdots s_{n-1})^\infty$.

\begin{theorem}\label{thm:bounds}
There are three possible cases.
\begin{itemize}
\item If $\sigma_0 = \sigma_{k-1} = +$, then $|\Allow_n(\Sigma_\sigma)| \leq a(n,k) + (k-2)k^{n-2}.$
\item If $\sigma_0 \neq \sigma_{k-1}$, then $|\Allow_n(\Sigma_\sigma)| \leq a(n,k) + (k-1)k^{n-2}.$
 \item If $\sigma_0 = \sigma_{k-1} = -$, then $|\Allow_n(\Sigma_\sigma)| \leq a(n,k) + (k^2-2)k^{n-3}.$
\end{itemize}
\end{theorem}
\begin{proof}
The idea of the proof is that $|\Allow_n(\Sigma_\sigma)|$ is bounded above by the number of allowed intervals for length $n$. We count the number of these allowed intervals and to get our upper bound. If $\sigma_0 = \sigma_{k-1} = +$, then $w_k = 0^\infty$ and $W_k = (k-1)^\infty$. If $s_{n-1} = 0$ or $k-1$, the endpoints $s_1\cdots s_{n-1} w_\sigma$ and $s_1\cdots s_{n-1}W_\sigma$ respectively, are already of the form $s_1\cdots s_t(s_{t+1} \cdots s_{n-1})^\infty$. We need to add one for every word $s_1\cdots s_{n-1}$ that does not end in 0 or $k-1$ to get an accurate count of the number of allowed intervals. There are $(k-2)k^{n-2}$ such words. 

If $\sigma_{0} = \sigma_{k-1} = -$, then $w_k = (0(k-1))^\infty$ and $W_k = ((k-1)0)^\infty$. If $s_{n-2}s_{n-1} = 0(k-1)$ or $(k-1)0$, then the endpoints $s_1\cdots s_{n-1} w_\sigma$ and $s_1\cdots s_{n-1}W_\sigma$ respectively, are already of the form $s_1\cdots s_t(s_{t+1} \cdots s_{n-1})^\infty$. Therefore, we only need to add one whenever $s_1\cdots s_{n-1}$ is such that $s_{n-2}s_{n-1} \neq 0(k-1)$ or $(k-1)0$. There are $(k^2-2)k^{n-3}$ such words.

If $\sigma_0 = +$ and $\sigma_{k-1} = -$, then $w_k = 0^\infty$ and $W_k = 10^\infty$. If $s_{n-1} = 0$ then $s_1\cdots s_{n-1} w_\sigma$ is already of the form $s_1\cdots s_t(s_{t+1} \cdots s_{n-1})^\infty$. So, we only need to add one when $s_{n-1} \neq 0$. There are $(k-1)k^{n-2}$ such words. Similarly for when $\sigma_0 = -$ and $\sigma_{k-1} = +$. 
\end{proof}

\subsection{Bounds for the tent map}

In the case of the tent map, $\tent = \Sigma_{+-}$, we can find slightly better bounds. First, we prove the following lemma. 
\begin{lemma}\label{lem: all but one}
If $\pi = \Pat(s, \tent, n) = \Pat(t, \tent, n)$, then $s_1\cdots s_{n-1} = t_1\cdots t_{n-1}$ except for up to one place.
\end{lemma}
\begin{proof}
Let $0 = e_0 \leq e_1 \leq e_2 = n$ be a $*$-$\sigma$-segmentation of $\pi$. Then $\pihat^*_1 < \cdots < \pihat^*_{e_1}$ and $\pihat^*_{e_1+1} > \cdots > \pihat^*_n$. If $*$ is not adjacent to the peak in $\pihat^*$ (position $j$ where $\pihat^*_{j-1}<\pihat^*_j>\pihat^*_{j+1}$), then we can choose $e_1$ in two ways, so that $\pihat^*_{e_1}<\pihat^*_{e_1+1}$ or that $\pihat^*_{e_1}>\pihat^*_{e_1+1}$. So, there is one possible place in $s_1\cdots s_{n-1}$ may change depending on which choice of $e_1$ we make. If $*$ is adjacent to the peak in $\pihat^*$, there may be more choices for $e_1$, but these will only affect whether $s_n$ is 0 or 1, not $s_1 \cdots s_{n-1}$.
\end{proof}

For the following theorem giving bounds on $|\Allow_n(\tent)|$, let $a_n := a(n,2)$. 
\begin{theorem}\label{thm:tent bounds}
The number of allowed patterns of the tent map satisfies these inequalities:
$$ \frac{1}{2}(a_n +2^{n-2}) \leq |\Allow_n(\tent)| \leq a_n - 2^{n-2}+1.$$
\end{theorem}
\begin{proof}
Let $\I_n$ denote the number of allowed intervals for length $n$. In this case, $\I_n = a_n + 2^{n-2}$. To prove the lower bound, we want to show that $|\Allow_n(\tent)| \geq \frac{1}{2}\I_n$.

To prove this, we first show that there are at most two allowed intervals with the same pattern. By Lemma \ref{lem: all but one}, we know that if $\Pat(s, \tent, n) = \Pat(t, \tent, n)$ then we must have the $s_1\cdots s_{n-1} = t_1\cdots t_{n-1}$ except for possibly one position. In addition, if $s$ and $t$ are words such that $s_1\cdots s_{n-1} = t_1\cdots t_{n-1}$ and $\Pat(s, \tent, n) = \Pat(t, \tent, n)$, then we claim that $s$ and $t$ must lie in the same interval. If not, then there would be some way to write $s_1\cdots s_{n-1} = up$ with $p = s_{r+1}\cdots s_{n-1}$ primitive, so that $s\prec_\tent up^\infty \prec_\tent t$. If $\nn{s_1\cdots s_{n-1}}$ is even (that is, it contains an even number of 1s), this implies that $s_{[n,\infty)} \prec_\tent p^\infty \prec_\tent t_{[n,\infty)}$. Notice $s_{[r+1, \infty)} = ps_{[n,\infty)}$. If we assume that $\pi_n <\pi_{r+1}$, then $$ps_{[n,\infty)} \prec_\sigma s_{[n,\infty)}\prec_\sigma p^\infty$$ which would imply that $s_{[n,\infty)} = p^\infty$. In this case, the pattern of $s$ would not be defined. Therefore, we must have that $\pi_{t+1}<\pi_n$. Similarly, we can show that because $\pi_{r+1}>\pi_n$ by comparing $t_{[r+1,\infty)}$ and $t_{[n,\infty)}$. The same is true when $\nn{s_1\ldots s_{n-1}}$ is odd, with inequalities $\prec_\sigma$ reversed. This is a contradiction, and thus we must have that $s$ and $t$ lie in the same interval.
This, along with the Lemma \ref{lem: all but one} imply that there are at most two intervals for each pattern. The lower bound follows.

To prove the upper bound, we use the fact that $\tent$ is continuous over its domain. By definition of the ordering on words $\prec_\tent$, we notice that $s_1\cdots s_{n-1}10^\infty = s_1\cdots \bar{s}_{n-1}10^\infty$ and that $s_1\cdots s_j 0^\infty= s_1\cdots \bar{s}_{j-1} s_j 0^\infty$, where $s_j=1$. Both of these follow from the fact that according to the ordering $\prec_\tent$, there are no words between $010^\infty$ and $110^\infty$. (This is simliar to the fact that according to lexicographical ordering on binary words, $01^\infty = 10^\infty$.) Notice that the pattern of $s_1\cdots s_{n-1}10^\infty = s_1\cdots \bar{s}_{n-1}10^\infty$ is defined. The pattern for these is the same as for their adjacent intervals, and so the two intervals that share this endpoint have the same pattern. This happens for all patterns that end in $\pi_n = n$. 
In the second case, $s_1\cdots s_j 0^\infty= s_1\cdots \bar{s}_{j-1} s_j 0^\infty$, where $s_j=1$, the patterns in the allowed intervals adjacent to this endpoint are the same and end in $n123\cdots m$ for some $m<n$. 

We use the fact that we know these patterns were paired up to get a better upper bound. At the endpoints of the domain, $s = 10^\infty$ and $s=0^\infty$, there is no adjacent interval, but at every other endpoint of the type $s_1\cdots s_{n-1} 0^\infty$ and $s_1\cdots s_{n-1}10^\infty$, there is this pairing. Therefore the number of pairings we have identified is $2^{n-1} -1$ since there are $2^{n-1}$ ways to write $s_1\cdots s_{n-1}$. So improve the upper bound to  $$|\Allow_n(\tent)|\leq \I_n - (2^{n-1}-1) = a_n -2^{n-2} +1.$$
\end{proof}

\begin{corollary}
For $n\geq 3$, $|\Allow_n(\tent)|<|\Allow_n(\Sigma_{++})|$. 
\end{corollary}
\begin{proof}
 By \cite{Elishifts}, $a_n = |\Allow_n(\Sigma_{++})|$. Therefore, Theorem \ref{thm:tent bounds} implies that for $n\geq3$, the number of allowed patterns for the tent map is strictly smaller than the number of allowed patterns for the binary shift. When $n=1$ or 2, we have $|\Allow_n(\tent)|=|\Allow_n(\Sigma_{++})|$.
 \end{proof}
 
Since the allowed patterns of $\Sigma_{+-}$ are complements of the allowed patterns of $\Sigma_{-+}$, we also have the following corollary.
 
 \begin{corollary}
The number of allowed patterns of $\Sigma_{-+}$ satisfies these inequalities:
$$ \frac{1}{2}(a_n +2^{n-2}) \leq |\Allow_n(\Sigma_{-+})| \leq a_n - 2^{n-2}+1.$$
\end{corollary}

There is a lot of room to improve these bounds, of course. The true value is much closer to $\frac{1}{2} \I_n$ than to $\I_n$ since most allowed intervals are paired up. Patterns $\pi$ so that there is a unique allowed interval with that pattern are exactly those patterns $\pi$ where $s_1\cdots s_{n-1}$ is determined uniquely. If $c_n$ denotes the number of these patterns, then $$|\Allow_n(\tent)| = \frac{1}{2}(\I_n-c_n) + c_n = \frac{1}{2}(\I_n + c_n)$$ so if we can enumerate these patterns, then we can also enumerate the allowed patterns of the tent map.
We characterize these patterns in the following theorem. However, because of their characterization, they are difficult to count.

\begin{theorem}
Allowed patterns $\pi$ of the tent map so that $s_1\cdots s_{n-1}$ is uniquely determined are those allowed patterns which admit some $*$-$\sigma$-segmentation $0=e_0\leq e_1\leq e_2 = n$ and some $b$ so that 
\begin{enumerate}
\item $\st(\pi_{n-2b}\pi_{n-b}\pi_n) = 312$ or $132$, 
\item $e_t< \pi_{n-2i}\leq e_{t+1}$ if and only if $e_t<\pi_{n-i}\leq e_{t+1}$ for all $1\leq i \leq b$ and for $t = 0,1$.
\end{enumerate}
\end{theorem}
\begin{proof}
For an allowed pattern $\pi\in \Allow_n(\tent)$, the number possible of $*$-$\sigma$-segmentations is restricted since by Lemma \ref{lem: all but one} there are only two possible ways to write $s_1\cdots s_{n-1}$. Notice that by the proof of Theorem \ref{thm:desc allowed}, the $*$-$\sigma$-segmentation described produces a word whose pattern is not defined. Since $\pi$ is allowed, there is some other $*$-$\sigma$-segmentation where this does not happen. However, since $\st(\pi_{n-2b}\pi_{n-b}\pi_n) = 312$ or $132$ and $e_t<\pi_{n-2b}, \pi_{n-b}\leq e_{t+1}$ for some $t$ in the $*$-$\sigma$-segmentation described above, the new $*$-$\sigma$-segmentation must change some letter in $s_1\cdots s_{n-1}$ (it is not enough for it to just change $s_n$). Therefore, $s_1\cdots s_{n-1}$ is uniquely determined by the new $*$-$\sigma$-segmentation. Conversely, if there is only one choice of $s_1\cdots s_{n-1}$ is the only choice for $\pi$, there must be a problem with the other choice. By Theorem \ref{thm:desc allowed}, we see that the problem that can arise in the case of the tent map is the one described in this theorem.
\end{proof}

\subsection{The $k$-shift}

Using the idea of allowed intervals, we are also able to recover the result from \cite{Elishifts} enumerating the number of allowed patterns of the $k$-shift. Certainly, $|\Allow_n(\Sigma_2)| = \I_n = a(n,2)$. Assume $k\geq 3$. Here, let $\I_{n,k}$ denote the number of allowed intervals for the $k$-shift:
$$ 
\I_{n,k} = \sum_{t=1}^{n-1} \psi_k(t) k^{n-t-1} + (k-2)k^{n-2}
$$
where $\psi_k(t)$ is the number of primitive words on $k$ letters of length $t$. Additionally, let us denote $$b(n,k) = |\Allow_n(\Sigma_k)\setminus\Allow_n(\Sigma_{k-1})|.$$
\begin{theorem}
For $n\geq3$ and $k\geq 3$, $$b(n,k) = \I_{n,k} - \sum_{i = 2}^{k-1} {{n+k-i}\choose{k-1}} b(n,i).$$
\end{theorem}
\begin{proof}
We prove this by showing that the following equivalent statement is true: $$ \I_{n,k} = \sum_{i = 2}^{k} {{n+k-i}\choose{k-1}} b(n,i).$$
First, suppose $s=s_1s_2\cdots$ and $t = t_1t_2\cdots$ so that $\Pat(s,\Sigma_k,n)= \Pat(t,\Sigma_k,n) = \pi$. Similarly to before, if $s_1\cdots s_{n-1}= t_1\cdots t_{n-1}$, then these two words must lie within the same allowed pattern. Also we don't have to worry about the condition $(\dagger)$ from Theorem \ref{thm:desc allowed} since $\sigma = +^k$. Therefore, any $*$-$\sigma$-segmentation of $\pihat$ will give a word whose pattern is $\pi$. 

We defined $b(n,i)$ to be the number of patterns realized by $\Sigma_i$ but not by $\Sigma_{i-1}$. These are patterns that have a $*$-$+^i$-segmentation but not a $*$-$+^{i-1}$-segmentation. So, either $\pihat$ has $i-1$ descents or it starts with $*1$ or ends with $n*$. In each case, there are ${{n+k-i}\choose{k-i}}$ possible $*$-$+^k$-segmentations of $\pihat$. Each of these is associated with a different allowed interval since for each of these we get a different $s_1\cdots s_{n-1}$. 
Since we consider all allowed patterns by considering each pattern counted by $b(n,i)$ for all $2\leq i \leq k$, we have counted every allowed interval in this way. Therefore, the left and right hand sides both count all allowed intervals.
\end{proof}
This theorem gives a recurrence for the allowed patterns of the $k$-shift that closely resembles the recurrence for the periodic patterns of the $k$-shift \cite{AE} and cyclic permutations with $k-1$ descents \cite{GesReut} and thus has the following generating function. 
\begin{corollary}
The numbers $b(n,k)$ satisfy the following generating function:
$$\frac{\sum_{k=2}^n b(n,k) x^k}{(1-x)^n} = \sum_{k\geq 1} \I_{n,k} x^k.$$
\end{corollary}
It follows that $b(n,k)$ has the same formula as presented in \cite{Elishifts}.

\subsection{Binary reverse shift}

In this section, we enumerate the allowed patterns for $\rs_2 = \Sigma_{--}$ exactly. Recall $a_n =a(n,2)$ is the number of allowed patterns for the binary shift $\Sigma_2$. 

\begin{theorem}
For $n\geq 3$, $|\Allow_n(\rs_2)| = a_n +2^{n-2}-2.$
\end{theorem}
\begin{proof}
For this proof, let $\prec \, := \prec_{\rs_2}$. The idea of this proof is to show that there are only two patterns where $s_1\cdots s_{n-1}$ is not determined uniquely and for each, there are only two possibilities for $s_1\cdots s_{n-1}$. Additionally, if two words have the same pattern and both start with $s_1\cdots s_{n-1}$, they must lie in the same interval. This would imply that the number of patterns is exactly $\I_n - 2$ where $\I_n$ is the number of allowed intervals. Using Theorem \ref{thm:bounds} to find $\I_n$, this would prove the formula above.

If $s$ and $t$ satisfy $s_1\cdots s_{n-1} = t_1\cdots t_{n-1}$ and $\Pat(s, \rs_2, n) = \Pat(t, \rs_2, n)$, then $s$ and $t$ must lie in the same interval. If not, then there would be some way to write $s_1\cdots s_{n-1} = up$ with $p = s_{t+1}\cdots s_{n-1}$, so that $s\prec up^\infty \prec t$ (or vice versa). If $n$ is odd, this implies that $s_{[n,\infty)} \prec p^\infty \prec t_{[n,\infty)}$. From this, we can see this implies that $\pi_n< \pi_{t+1}<\pi_n$. Similarly if $n$ is even (the inequalities are switched). 

Consider four cases, depending on $\pihat^*$. 

 {\it Case 1.}  Suppose $\pihat^*$ has a clear ascent. That is, there is some $i$ so that $\pihat^*_i<\pihat^*_{i+1}$. Then $s_1\cdots s_{n-1}$, a $\pi$-monotone word, is clearly forced. 

{\it Case 2.} Suppose there is some $i$ so that $\pihat^*_i<\pihat^*_{i+2}$ and $\pi_{i+1} = *$. Then $s_1\cdots s_{n-1}$ are still forces, but $s_n$ may be 0 or 1. 

{\it Case 3.} Suppose that there is some $i \in [2, n-1]$ so that $\pihat^*_i = *$ and $\pihat^*_1>\pihat^*_{2}>\cdots >\pihat^*_{i-1}>\pihat^*_{i+1}>\cdots > \pihat^*_n.$ Here there is no clear ascent, but we will show that $s_1\cdots s_{n-1}$ will still be determined. First notice that $\pihat_i = n$ or 1, that is, either the number $n$ or 1 does not appear in $\pihat^*$. If not, then we would have $\pihat^*_1 = n$ and $\pihat^*_n = 1$, but this would imply that $n$ comes after 1 in $\pi$ and that $1$ comes after $n$. This is a contradiction. Therefore, the numbers $\pihat^*_1>\pihat^*_{2}>\cdots >\pihat^*_{i-1}>\pihat^*_{i+1}>\cdots > \pihat^*_n$ are actually the consecutive numbers $[1,n-1]$ or $[2,n]$. 

Another thing to notice is that if $n$ is odd, $i = (n+1)/2$ and if $n$ is even, $i = n/2$ or $n/2 + 1$, that is, the $*$ is in the middle of $\pihat^*$. This is needed in order for $\pihat^*$ to be cyclic. Obtaining $\pi$ from $\pihat^*$, we get that $\pi = 1n2(n-1)3(n-2)\cdots$ or $\pi = n1(n-1)2(n-2)3\cdots$, so the last position is $\pi_n = (n+1)/2$ if $n$ is odd, $\pi_n = n/2 + 1$ if $n$ is even and starts with $\pi_1 =1$, and $\pi_n = n/2$ if $n$ is even and starts with $\pi_1 =n$.

Suppose $\pihat^*_{i-1} = j$ and $\pihat^*_{i+1} = j-1$. Then this implies that in $\pi$, we have for some $a$, $\pi_a\pi_{a+1} = (i-1)j$, for some $b$, $\pi_b\pi_{b+1} = (i+1)(j-1)$ and $\pi_n = i$. If $n$ is odd and $\pi_1 = 1$ (so that $\pihat^*$ contains numbers from $[2,n-1]$), then $j-1 = (n+1)/2$, but this is also equal to $i$. Therefore, $j-1 = i$, and $j = i+1$. This implies that $b = n-1$ and $a = n-2$. So we have $\pi_{n-2}\pi_{n-1}\pi_n = (i-1)(i+1)i$, which is a 132 pattern. By ($\dagger$) of Theorem \ref{thm:desc allowed}, we must let $s_{n-2} = 0$ and $s_{n-1} = 1$. The rest of $s_1\cdots s_{n-1}$ is forced after this. If $n$ is odd and $\pi_1 = n$ (so that $\pihat^*$ contains numbers from $[1,n]$), then $j = i$ and we end up with $\pi_{n-2}\pi_{n-1}\pi_n = (i+1)(i-1)i$ and so by ($\dagger$) of Theorem \ref{thm:desc allowed}, we must let $s_{n-2} = 1$ and $s_{n-1} = 0$ and as before, the rest of $s_1\cdots s_{n-1}$ is forced after this. Similarly for when $n$ is even.

{\it Case 4.} Now suppose we have either $\pihat^*_1 = *$ or $\pihat^*_n = *$ and that $\pihat^*$ is decreasing (no ascents). Because $\pihat^*$ is cyclic, $\pihat^*$ is forced to contain numbers $[1, n] \setminus \{(n+1)/2\}$ when $n$ is odd,  $[1, n] \setminus \{(n)/2\}$ when $n$ is even and $\pihat^*_n = *$, and $[1, n] \setminus \{(n)/2-1\}$ when $n$ is even and $\pihat^*_1 = *$.  So there are certainly on two such permutations, the one you get from when $\pihat^*_1 = *$ and from when $\pihat^*_n = *$. It remains to show that there are only two possible $s_1\cdots s_{n-1}$ for each. 

Notice that we must have either (4) or (5) of the definition of a $*$-$\sigma$-segmentation. If $\pihat^*_1 = *$, then we also have $\pihat^*_2 = n$ and $\pihat^*_n = 1$. Therefore we must have that $e_1 = n$ or $e_1 \leq 1$. Therefore, $s_1\cdots s_{n-1} = 0^\infty$ or $1^\infty$ (if $e_1 = 0$, then $s_n = 1$ and if $e_1 = 1$, then $s_n = 0$, but these won't change the rest of the word). Similarly, if $\pihat^*_n = *$, then we also have $\pihat^*_{n-1} = 1$ and $\pihat^*_1 = n$. Therefore we must have that $e_1 = 0$ or $e_1 \geq n-1$. Therefore, $s_1\cdots s_{n-1} = 0^\infty$ or $1^\infty$. 

Therefore, there are only two patterns that have two possibilities for $s_1\cdots s_{n-1}$ and all others determine $s_1\cdots s_{n-1}$ uniquely.
\end{proof}

\begin{corollary}
For $n\geq 4$, $|\Allow_n(\rs_2)|>|\Allow_n(\Sigma_2)|$. When $n\leq 3$, they are equal.
\end{corollary}
\begin{proof}
This follows from the theorem above and that $|\Allow_n(\Sigma_2)| = a_n$. 
\end{proof}

Based on numerical evidence, we conjecture the following generalization of the above corollary.
\begin{conjecture}
For $n\geq 3$, $|\Allow_n(\Sigma_k)|\leq |\Allow_n(\rs_k)|$. 
\end{conjecture}
Numerical evidence suggests an even stronger conjecture, namely that for any $\sigma \in \{+,-\}^k$, which is not equal to $+^k$ or $-^k$, we have $$|\Allow_n(\Sigma_\sigma)|\leq |\Allow_n(\Sigma_k)|\leq |\Allow_n(\rs_k)|.$$

\end{document}